\documentclass{article}
\usepackage{amssymb,amsmath,amsfonts}

\newtheorem{theorem}{Theorem}

\newtheorem{corollary}[theorem]{Corollary}

\newtheorem{definition}[theorem]{Definition}
\newtheorem{example}[theorem]{Example}

\newtheorem{lemma}[theorem]{Lemma}

\newtheorem{proposition}[theorem]{Proposition}
\newtheorem{remark}[theorem]{Remark}

\newenvironment{proof}[1][Proof]{\noindent\textbf{#1.} }{\ \rule{0.5em}{0.5em}}

\begin{document}

\title{Some examples of BL-algebras using commutative rings}
\author{Cristina Flaut and Dana Piciu}
\date{}
\maketitle

\begin{abstract}
BL-algebras are algebraic structures corresponding to Hajek's basic fuzzy
logic.

The aim of this paper is to analize the structure of BL-algebras using
commutative rings. From computational considerations, we are very interested
in the finite case. We present new ways to generate finite BL-algebras using
commutative rings and we give summarizing statistics.

Furthermore, we investigated BL-rings, i.e., commutative rings whose the
lattice of ideals can be equiped with a structure of BL-algebra. A new
characterization for these rings and their connections to other classes of
rings are established. Also, we give examples of finite BL-rings for which
their lattice of ideals is not an MV-algebra and using these rings we
construct BL-algebras with $2^{r}+1$ elements, $r\geq 2,$ and \ all
BL-chains with $k$ elements, $k\geq 4.$


\end{abstract}

\section{\textbf{Introduction}}

The origin of residuated lattices is in Mathematical Logic. They have been
introduced by Dilworth and Ward, through the papers \cite{[Di; 38]}, \cite%
{[WD; 39]}.\medskip\ The study of residuated lattices is originated in 1930
in the context of theory of rings, with the study of ring ideals. It is
known that the lattice of ideals of a commutative ring is a residuated
lattice, see \cite{[Bl; 53]}. Several researchers (\cite{[Bl; 53]}, \cite%
{[BN; 09]}, \cite{[LN; 18]}, \cite{[TT; 22]}, etc) have been interested in
this construction.

Two important subvarieties of residuated lattices are BL-algebras
(corresponding to Hajek's logic, see \cite{[H; 98]}) and MV-algebras
(corresponding to \L ukasiewicz many-valued logic, see \cite{[COM; 00]}, 
\cite{[CHA; 58]}). For instance, rings for which the lattice of ideals is a
BL-algebra are called BL-rings and are introduced in \cite{[LN; 18]}.

In this paper, we obtain a description for BL-rings, using a new
characterization of BL-algebras, given in Theorem \ref{Theorem_1}, i.e,
residuated lattices $L$ in which $\ [x\odot (x\rightarrow y)]\rightarrow
z=(x\rightarrow z)\vee (y\rightarrow z)$ for every $x,y,z\in L.$ Then,
BL-rings are unitary and commutative rings $A$ with the property that $K:$%
\textit{\ }$[I\otimes (J:I)]=(K:I)+(K:J),$ \textit{for every }$I,J,K\in
Id(A),$\ see Corollary \ref{Corollary 3.6}

Also, we show that the class of BL-rings contains other known classes of
commutative rings: rings which are principal ideal domains and some types of
finite unitary commutative rings, see Theorem \ref{Theorem_2}, Corollary \ref%
{Corollary_2_1} and Corollary \ref{Corollary _3_4}.

One of recent applications of BCK algebras is given by Coding Theory. \ From
computational considerations, we are interested, in this paper, to find ways
to generate finite BL-algebras using finite commutative rings, since a
solution that is computational tractable is to consider algebras with a
small number of elements. First, we give examples of finite BL-rings whose
lattice of ideals is not an MV-algebra. Using these rings we construct
BL-algebras with $2^{r}+1$ elements, $r\geq 2$ (see Proposition \ref%
{Proposition10}) and BL-chains with $k\geq 4$ elements (see Proposition \ref%
{Proposition11}). Also, in Theorem \ref{Theorem_4}, we present a way to
generate all (up to an isomorphism) finite BL-algebras with $2\leq n\leq 5$
elements, by using the ordinal product of residuated lattices and we present
summarizing statistics.

\section{\textbf{Preliminaries}}

\begin{definition}
\label{Definition_1} (\cite{[Di; 38]}, \cite{[WD; 39]})\textbf{\ \ }A \emph{%
(commutative) residuated lattice} \ is an algebra $(L,\wedge ,\vee ,\odot
,\rightarrow ,0,1)$ such that:

(LR1) $\ (L,\wedge ,\vee ,0,1)$ is a bounded lattice;

(LR2) $\ \ (L,\odot ,1)$ is a commutative ordered monoid;

(LR3) $\ z\leq x\rightarrow y$\ iff $x\odot z\leq y,$\ for all $x,y,z\in L.$
\end{definition}

The property (LR3) is called\emph{\ residuation}, where $\leq $ is the
partial order of the lattice $\ (L,\wedge ,\vee ,0,1).$

In a residuated lattice is defined an additional operation: for $x\in L,$ we
denote $x^{\ast }=x\rightarrow 0.$

\begin{example}
\label{Example_1} (\cite{[T; 99]}) Let $(\mathcal{B},\wedge ,\vee ,^{\prime
},0,1)$ be a \emph{Boolean algebra.} \ If we define for every $x,y\in 
\mathcal{B},x\odot y=x\wedge y$ and $x\rightarrow y=x^{\prime }\vee y,$ then 
$(\mathcal{B},\wedge ,\vee ,\odot ,\rightarrow ,0,1)$ becomes a residuated
lattice.
\end{example}

\begin{example}
\label{Example_2} It is known that, for a commutative unitary ring, $A,$ if
we denote by $Id\left( A\right) $ the set of all ideals, then for $I,J\in
Id\left( A\right) $, the following sets 
\begin{equation*}
I+J=<I\cup J>=\{i+j,i\in I,j\in J\}\text{, }
\end{equation*}%
\begin{equation*}
I\otimes J=\{\underset{k=1}{\overset{n}{\sum }}i_{k}j_{k},\text{ }i_{k}\in
I,j_{k}\in J\}\text{, }
\end{equation*}%
\begin{equation*}
\left( I:J\right) =\{x\in A,x\cdot J\subseteq I\}\text{,}
\end{equation*}%
\begin{equation*}
Ann\left( I\right) =\left( \mathbf{0}:I\right) \text{, where }\mathbf{0}=<0>,%
\text{ }
\end{equation*}%
\textbf{\ }are also ideals of $A,$ called sum, product, quotient and
annihilator, see \cite{[BP; 02]}. If we preserve these notations, $%
(Id(A),\cap ,+,\otimes \rightarrow ,0=\{0\},1=A)$ is a residuated lattice in
which the order relation is $\subseteq $ and $I\rightarrow J=(J:I),$ for
every $I,J\in Id(A)$, see \cite{[TT; 22]}.
\end{example}

In a residuated lattice $(L,\wedge ,\vee ,\odot ,\rightarrow ,0,1)$ we
consider the following identities:

\begin{equation*}
(prel)\qquad (x\rightarrow y)\vee (y\rightarrow x)=1\qquad \text{ (\textit{%
prelinearity)}};
\end{equation*}%
\begin{equation*}
(div)\qquad x\odot (x\rightarrow y)=x\wedge y\qquad \text{ (\textit{%
divisibility)}}.
\end{equation*}

\begin{definition}
\label{Definition_2} (\cite{[NL; 03]}, \cite{[I; 09]}, \cite{[T; 99]}) \ A
residuated lattice $L$ is called \emph{a BL-algebra}\textit{\ }if $L$
verifies $(prel)+(div)$ conditions.
\end{definition}

A \emph{BL-chain }is a totally ordered BL-algebra, i.e., a BL-algebra such
that its lattice order is total.

\begin{definition}
\label{Definition_3} (\cite{[COM; 00]}, \cite{[CHA; 58]}) An \emph{MV-algebra%
} is an algebra $\left( L,\oplus ,^{\ast },0\right) $ satisfying the
following axioms:

(MV1) $\left( L,\oplus ,0\right) $ \ is an abelian monoid;

(MV2) $(x^{\ast })^{\ast }=x;$

(MV3) $x\oplus 0^{\ast }=0^{\ast };$

(MV4) $\left( x^{\ast }\oplus y\right) ^{\ast }\oplus y=$ $\left( y^{\ast
}\oplus x\right) ^{\ast }\oplus x$, for all $x,y\in L.$
\end{definition}

In fact, a residuated lattice $L$ is an MV-algebra iff it satisfies the
additional condition: 
\begin{equation*}
(x\rightarrow y)\rightarrow y=(y\rightarrow x)\rightarrow x,
\end{equation*}%
for every $x,y\in L,$ see \cite{[T; 99]}.

\begin{remark}
\label{Remark_1} (\cite{[T; 99]}) \ If in a BL- algebra $L$, $x^{\ast \ast
}=x,$ for every $x\in L$, and for $x,y\in L$ we denote 
\begin{equation*}
x\oplus y=(x^{\ast }\odot y^{\ast })^{\ast },
\end{equation*}%
then we obtain an MV-algebra $(L,\oplus ,^{\ast },0).$ Conversely, if $%
(L,\oplus ,^{\ast },0)$ is an MV\textit{-algebra}, then $(L,\wedge ,\vee
,\odot ,\rightarrow ,0,1)$ becomes a BL-algebra, in which for $x,y\in L:$ 
\begin{equation*}
x\odot y=(x^{\ast }\oplus y^{\ast })^{\ast },
\end{equation*}%
\begin{equation*}
x\rightarrow y=x^{\ast }\oplus y,1=0^{\ast },
\end{equation*}%
\begin{equation*}
x\vee y=(x\rightarrow y)\rightarrow y=(y\rightarrow x)\rightarrow x\text{
and }x\wedge y=(x^{\ast }\vee y^{\ast })^{\ast }.
\end{equation*}%
In fact, MV-algebras are exactly involutive BL-algebras.
\end{remark}

\begin{example}
\label{Example_3}\emph{(}\cite{[I; 09]}\emph{) }\textit{\ }We give an
example of a finite BL-algebra which is not an MV-algebra. Let $%
L=\{0,a,b,c,1\}$ and we define on $L$ the following operations: 
\begin{equation*}
\begin{array}{c|ccccc}
\rightarrow & 0 & c & a & b & 1 \\ \hline
0 & 1 & 1 & 1 & 1 & 1 \\ 
c & 0 & 1 & 1 & 1 & 1 \\ 
a & 0 & b & 1 & b & 1 \\ 
b & 0 & a & a & 1 & 1 \\ 
1 & 0 & c & a & b & 1%
\end{array}%
,\hspace{5mm}%
\begin{array}{c|ccccc}
\odot & 0 & c & a & b & 1 \\ \hline
0 & 0 & 0 & 0 & 0 & 0 \\ 
c & 0 & c & c & c & c \\ 
a & 0 & c & a & c & a \\ 
b & 0 & c & c & b & b \\ 
1 & 0 & c & a & b & 1%
\end{array}%
.
\end{equation*}
\end{example}

We have, $0\leq c\leq a,b\leq 1,$ but $a,b$ are incomparable, hence $L$ is a
BL-algebra which is not a chain. We remark that $x^{\ast \ast }=1,$ for
every $x\in L,x\neq 0.$

\section{BL-rings}

It is known that if\textbf{\ } $A$ is a commutative unitary ring, then $%
(Id(A),\cap ,+,\otimes \rightarrow ,0=\{0\},1=A)$ is a residuated lattice,
see \cite{[TT; 22]}.

\begin{definition}
\label{Definition_4} (\cite{[LN; 18]}) \ A commutative ring whose lattice of
ideals is a BL-algebra is called \emph{a BL-ring.}
\end{definition}

BL-rings are closed under finite direct products, arbitrary direct sums and
homomorphic images, see \cite{[LN; 18]}.

In the following, using the connections between BL-algebras and BL-rings, we
give a new characterizations for commutative and unitary rings for which the
lattice of ideals is a BL-algebra.

\begin{proposition}
\label{Proposition_3} (\cite{[I; 09]}) \textit{\ Let }$(L,\vee ,\wedge
,\odot ,\rightarrow ,0,1)$\textit{\ be a residuated lattice. Then we have
the equivalences:}

\textit{(i) }$L\ $satisfies $(prel)$\textit{\ condition;}

\textit{(ii) }$\ \ (x\wedge y)\rightarrow z=(x\rightarrow z)\vee
(y\rightarrow z),$ \textit{for every \thinspace }$x,y,z\in L.$
\end{proposition}

\begin{lemma}
\label{Lemma_0} \textit{\ Let }$(L,\vee ,\wedge ,\odot ,\rightarrow ,0,1)$%
\textit{\ be a residuated lattice. The following assertions are equivalent:}

\textit{(i) }$L$ satisfies $(prel)$\textit{\ condition;}

\textit{(ii) }$\ \ $F\textit{or every \thinspace }$x,y,z\in L,$ if $\
x\wedge y\leq z,$ then $(x\rightarrow z)\vee (y\rightarrow z)=1.$\textit{\ }
\end{lemma}

\begin{proof}
$(i)\Rightarrow (ii).$ Following by Proposition \ref{Proposition_3}.

$(ii)\Rightarrow (i).$ $\ $Using (ii), for $z=$ $x\wedge y$ we deduce that $%
1=(x\rightarrow (x\wedge y))\vee (y\rightarrow (x\wedge y))=[(x\rightarrow
x)\wedge (x\rightarrow y)]\vee \lbrack (y\rightarrow x)\wedge (y\rightarrow
y)]$ $=(x\rightarrow y)\vee (y\rightarrow x),$ so, $L$ satisfies $(prel)$%
\textit{\ condition.}
\end{proof}

\begin{lemma}
\label{Lemma_1} \textit{\ Let }$(L,\vee ,\wedge ,\odot ,\rightarrow ,0,1)$%
\textit{\ be a residuated lattice. The following assertions are equivalent:}

\textit{(i) }$L$ satisfies $(div)$\textit{\ condition;}

\textit{(ii) }$\ \ $F\textit{or every \thinspace }$x,y,z\in L,$ if $x\odot
(x\rightarrow y)\leq z,$ then $x\wedge y\leq z.$\textit{\ }
\end{lemma}

\begin{proof}
$(i)\Rightarrow (ii).$ Obviously.

$(ii)\Rightarrow (i).$ $\ $Using (ii), for $z=$ $x\odot (x\rightarrow y)$ we
deduce that $x\wedge y\leq x\odot (x\rightarrow y).$ Since in a residuated
lattice, $x\odot (x\rightarrow y)\leq x\wedge y$ we deduce that $L$
satisfies $(div)$\textit{\ condition.}
\end{proof}

Using Lemma \ref{Lemma_0} and Lemma \ref{Lemma_1} we deduce that:

\begin{proposition}
\label{Proposition_4} \textit{Let }$(L,\vee ,\wedge ,\odot ,\rightarrow
,0,1) $\textit{\ be a residuated lattice. The following assertions are
equivalent:}

(i) $L$ is a BL-algebra;

(ii) For every $x,y,z\in L,$ if $x\odot (x\rightarrow y)\leq z,$ then $%
(x\rightarrow z)\vee (y\rightarrow z)=1;$

(iii) $[x\odot (x\rightarrow y)]\rightarrow z=(x\rightarrow z)\vee
(y\rightarrow z),$ for every $x,y,z\in L.$
\end{proposition}

\begin{proof}
$(i)\Rightarrow (ii).$ Let $x,y,z\in L$ such that $x\odot (x\rightarrow
y)\leq z.$ Since every BL-algebra satisfies $(div)$\textit{\ condition, }by
Lemma \ref{Lemma_1}, we deduce that $x\wedge y\leq z.$ Since every
BL-algebra satisfies $(prel)$\textit{\ condition, }following Lemma \ref%
{Lemma_0}, we deduce that $1=(x\rightarrow z)\vee (y\rightarrow z).$

$(ii)\Rightarrow (i).$ First, we prove that $L$ satisfies condition (ii)
from Lemma \ref{Lemma_0}. So, let $x,y,z\in L$ such that $x\wedge y\leq z.$
Thus, $(x\wedge y)\rightarrow z=1$ \ Since $x\odot (x\rightarrow y)\leq
x\wedge y,$ we deduce that $1=(x\wedge y)\rightarrow z\leq (x\odot
(x\rightarrow y))\rightarrow z.$ Then, $x\odot (x\rightarrow y)\leq z.$ By
hypothesis, $(x\rightarrow z)\vee (y\rightarrow z)=1.$

To prove that $L$ verifies condition $(ii)$ from Lemma \ref{Lemma_1}, let $%
x,y,z\in L$ such that $x\odot (x\rightarrow y)\leq z.$ By hypothesis, we
deduce that, $(x\rightarrow z)\vee (y\rightarrow z)=1.$ Since $(x\rightarrow
z)\vee (y\rightarrow z)\leq (x\wedge y)\rightarrow z$ we obtain $(x\wedge
y)\rightarrow z=1,$ that is, $x\wedge y\leq z.$\textit{\ }

$(iii)\Rightarrow (ii).$ Obviously.

$(ii)\Rightarrow (iii).$ If we denote $t=[x\odot (x\rightarrow
y)]\rightarrow z,$ we have $1=t\rightarrow t=t\rightarrow \lbrack (x\odot
(x\rightarrow y))\rightarrow z]=[x\odot (x\rightarrow y)]\rightarrow
(t\rightarrow z)$, hence, $x\odot (x\rightarrow y)\leq t\rightarrow z.$

By hypothesis, we deduce that, $(x\rightarrow (t\rightarrow z))\vee
(y\rightarrow (t\rightarrow z))=1.$

Then $1=(t\rightarrow (x\rightarrow z))\vee (t\rightarrow (y\rightarrow
z))\leq t\rightarrow \lbrack (x\rightarrow z)\vee (y\rightarrow z)].$ Thus, $%
t\leq (x\rightarrow z)\vee (y\rightarrow z).$

But $(x\rightarrow z)\vee (y\rightarrow z)\leq (x\wedge y)\rightarrow z\leq
\lbrack x\odot (x\rightarrow y)]\rightarrow z=t.$

We conclude that $t=(x\rightarrow z)\vee (y\rightarrow z),$ that is, $%
[x\odot (x\rightarrow y)]\rightarrow z=(x\rightarrow z)\vee (y\rightarrow
z), $ for every $x,y,z\in L.$
\end{proof}

Using Proposition \ref{Proposition_4} we obtain a new characterisation for
BL-algebras:

\begin{theorem}
\label{Theorem_1}\textit{A residuated lattice }$L$ is a BL-algebra if and
only if for every $x,y,z\in L,$ 
\begin{equation*}
\lbrack x\odot (x\rightarrow y)]\rightarrow z=(x\rightarrow z)\vee
(y\rightarrow z).
\end{equation*}
\end{theorem}

Using this result, we can give a new description for BL-rings:

\begin{corollary}
\label{Corollary 3.6}\textit{\ Let }$A$\textit{\ be a commutative and
unitary ring. The following assertions are equivalent:}

\textit{(i) }$\ A$ \textit{is a BL-ring;}

\textit{(ii) \ }$K:$\textit{\ }$[I\otimes (J:I)]=(K:I)+(K:J),$ \textit{for
every }$I,J,K\in Id(A).$
\end{corollary}

\begin{theorem}
\label{Theorem_2}\textit{Let }$A$\textit{\ be a commutative ring which is a
principal ideal domain. Then} $A$\textit{\ is a BL-ring.}
\end{theorem}

\begin{proof}
Since $A$ is a principal ideal domain, let $I=<a>$, $J=<b>,$ be the
principal non-zero ideals generated by $a,b\in A$ $\backslash \{0\}$.

If $d=$\textit{gcd}$\{a,b\}$, then $d=a\cdot \alpha +b\cdot \beta ,$ $a,b\in
A$, $a=a_{1}d$ and $b=b_{1}d$, with $1=$\textit{gcd}$\{a_{1},b_{1}\}$. Thus, 
$I+J=<d>,$ $\ I\cap J=<ab/d>,$ $I\otimes J=<ab>$ and $\left( I:J\right) =$ $%
<a_{1}>.$

The conditions $(prel)$ is satisfied: $(I:J)+(J:I)=<a_{1}>+<b_{1}>=<1>=A$
and $(div)$ is also satisfied: $J\otimes (I:J)=<b>\otimes
<a_{1}>=<ab/d>=I\cap J$.

If $I=\{0\}$, since $A$ is an integral domain, we have that $\left( \mathbf{0%
}:J\right) +(J:\mathbf{0})=Ann(J)+A=A$ and $J\otimes (\mathbf{0}:J)=J\otimes
Ann(J)=\mathbf{0=0}\cap J=$ $\mathbf{0}\otimes (J:\mathbf{0}),$ for every $%
J\in Id\left( A\right) \backslash \{0\}.$

Moreover, we remark that $\left( I:\left( I:J\right) \right) =\left(
J:\left( J:I\right) \right) =I+J$, for every non-zero ideals \textit{\ }$%
I,J\in Id(A).$ Also, since $A$ is an integral domain, we obtain $%
Ann(Ann(I))=A,$ for every $I\in Id\left( A\right) \backslash \{0\}.$ We
conclude that $Id(A)$ is a BL-algebra which is not an MV-algebra.
\end{proof}

\begin{corollary}
\label{Corollary_2_1}\textit{A ring factor of a principal ideal domain is a
BL-ring. }
\end{corollary}

\begin{proof}
Using Theorem \ref{Theorem_2} since BL-rings are closed under homomorphic
images, see \cite{[LN; 18]}. Moreover, we remark that, in fact, if $\mathcal{%
A}$ \textit{is a ring factor of a principal ideal domain, then }$(Id(%
\mathcal{A}),\cap ,+,\otimes \rightarrow ,0=\{0\},1=A)$ \textit{\ }is an
MV-algebra, see \cite{[FP; 22]}.
\end{proof}

\begin{corollary}
\label{Corollary _3_4} \textit{\ A finite commutative unitary ring of the
form }$A=\mathbb{Z}_{k_{1}}\times \mathbb{Z}_{k_{2}}\times ...\times \mathbb{%
Z}_{k_{r}}$(direct product of rings, equipped with componentwise
operations), \textit{where} $k_{i}=p_{i}^{\alpha _{i}}$, $p_{i}$ \textit{is} 
\textit{a prime number, with }$p_{i}\neq p_{j},$\textit{\ is a BL-ring.}
\end{corollary}

\begin{proof}
We apply Corollary \ref{Corollary_2_1} using the fact that BL-rings are
closed under finite direct products, see \cite{[LN; 18]}.

Moreover, we remark that if $\ A$ is a finite commutative unitary ring of
the above form, then $Id\left( A\right) =$ $Id(\mathbb{Z}_{k_{1}})\times Id(%
\mathbb{Z}_{k_{2}})\times ...\times Id(\mathbb{Z}_{k_{r}})$\ is an
MV-algebra $(Id\left( A\right) ,\oplus ,^{\ast },0=\{0\})$ in which%
\begin{equation*}
I\oplus J=Ann(Ann(I)\otimes Ann(J))\text{ and }I^{\ast }=Ann(I)
\end{equation*}%
for every $I,J\in Id\left( A\right) $, since, $Ann(Ann(I))=I,$ see \cite%
{[FP; 22]}.
\end{proof}

\begin{example}
\label{Remark_2} 1) Following Theorem \ref{Theorem_2}, the ring of integers $%
(\mathbb{Z},\mathbb{+},\cdot )$ is a BL-ring in which $(Id(\mathbb{Z}),\cap
,+,\otimes \rightarrow ,0=\{0\},1=A)$ is not an MV-algebra. Indeed, since $%
\mathbb{Z}$ is principal ideal domain, we have $Ann\left( Ann\left( I\right)
\right) =\mathbb{Z}$, for every $I\in Id\left( \mathbb{Z}\right) \backslash
\{0\}.$

2) Let $K$ be a field and $K\left[ X\right] $ be the polynomial ring. For $%
f\in K\left[ X\right] $, the quotient ring $A=K\left[ X\right] /\left(
f\right) $ is a BL-ring. Indeed, the lattice of ideals of this\ ring is an
MV-algebra, see \cite{[FP; 22]}.
\end{example}

\section{Examples of BL-algebras using commutative rings}

In this section we present ways to generate finite\ BL-algebras using finite
commutative rings.

First, we give examples of finite BL-rings whose lattice of ideals is not an
MV-algebra. Using these rings we construct BL-algebras with $2^{r}+1$
elements, $r\geq 2$ (see Proposition \ref{Proposition10}) and BL-chains with 
$k\geq 4$ elements (see Proposition \ref{Proposition11}).

We recall that, in \cite{[FP; 22]}, we proved the following result:

\begin{proposition}
\label{Proposition1} (\cite{[FP; 22]}\textit{\ If} $A$ \textit{is a finite
commutative unitary ring of the form }$A=\mathbb{Z}_{k_{1}}\times \mathbb{Z}%
_{k_{2}}\times ...\times \mathbb{Z}_{k_{r}}$(direct product of rings,
equipped with componentwise operations), \textit{where} $k_{i}=p_{i}^{\alpha
_{i}}$, $p_{i}$ \textit{is} \textit{a prime number, with }$p_{i}\neq p_{j},$%
\textit{\ and} $Id\left( A\right) $ \textit{denotes the set of all ideals of
the ring} $A$, \textit{then,} $\left( Id\left( A\right) ,\vee ,\wedge ,\odot
,\rightarrow ,0,1\right) $ \textit{is an MV-algebra, where the order
relation is} $\subseteq $, $I\odot J=I\otimes J$, $I^{\ast }=Ann(I)$, $%
I\rightarrow J=\left( J:I\right) $, $I\vee J=I+J$, $I\wedge J=I\cap J$, $%
0=\{0\}$ \textit{and} $1=A$. \textit{The set} $Id\left( A\right) $ \textit{%
has} $\mathcal{N}_{A}=\overset{r}{\underset{i=1}{\prod }}\left( \alpha
_{i}+1\right) $ \textit{elements}.
\end{proposition}

In the following, we give examples of finite BL-rings whose lattice of
ideals is not an MV-algebra.\medskip 

\begin{definition}
\label{Definition2} (\cite{[BP; 02]}) Let $R$ be a commutative unitary ring.
The ideal $M$ of the ring $R$ is \textit{maximal} if it is maximal, with
respect of the set inclusion, amongst all proper ideals of the ring $R$.
That means, there are no other ideals different from $R$ contained $M$. The
ideal $J$ of the ring $R$ is a \textit{minimal ideal} if it is a nonzero
ideal which contains no other nonzero ideals. A commutative \textit{local
ring} $R$ is a ring with a unique maximal ideal.
\end{definition}

\begin{example}
\label{Example3} (i) A field $F$ is a local ring, with $\{0\}$ the maximal
ideal in this ring.

(ii) In $\left( \mathbb{Z}_{8},+,\cdot \right) $, the ideal $J=\{\widehat{0},%
\widehat{4}\}$ is a minimal ideal and the ideal $M=\{\widehat{0},\widehat{2},%
\widehat{4},\widehat{6}\}$ is the maximal ideal.
\end{example}

\begin{remark}
\label{Remark4} Let $R$ be a local ring with $M$ its maximal ideal. Then,
the quotient ring $R[X]/(X^{n})$, with $n$ a positive integer, is local.
Indeed, the unique maximal ideal of the ring $R[X]/(X^{n})$ is $%
\overrightarrow{M}=\{\overrightarrow{f}\in R[X]/(X^{n})/~f\in
R[X],f=a_{0}+a_{1}X+...+a_{n-1}X^{n-1}$, with $a_{0}\in M\}$. For other
details, the reader is referred to \cite{[La; 01]}.
\end{remark}

In the following, we consider the ring $(\mathbb{Z}_{n},+,\cdot )$ with $%
n=p_{1}p_{2}...p_{r},p_{1},p_{2},...,p_{r}$ being distinct prime numbers, $%
r\geq 2,$ and the factor ring $R=\mathbb{Z}_{n}[X]/\left( X^{2}\right) .$

\begin{remark}
\label{Remark5} (i)\textbf{\ }With the above notations, in the ring $(%
\mathbb{Z}_{n},+,\cdot )$, the ideals generated by $\widehat{p}%
_{i},M_{p_{i}}=\left( \widehat{p}_{i}\right) ,$ are maximals. The ideals of $%
\mathbb{Z}_{n}$ are of the form $I_{d}=\left( \widehat{d}\right) $, where $d$
is a divisor of $n$.

(ii) Each element from $\mathbb{Z}_{n}-\{M_{p_{1}}\cup M_{p_{2}}\cup ...\cup
M_{p_{r}}\}$ is an invertible element. Indeed, if $\widehat{x}\in \mathbb{Z}%
_{n}-\{M_{p_{1}}\cup M_{p_{2}}\cup ...\cup M_{p_{r}}\}$, we have \textit{gcd}%
\thinspace $\{x,n\}=1$, therefore $x$ is an invertible element.
\end{remark}

\begin{proposition}
\label{Proposition6} \textit{i)} \textit{With the above notations, the
factor ring} $R=\mathbb{Z}_{n}[X]/\left( X^{2}\right) $ \textit{has} $2^{r}+1
$ \textit{ideals, including} $\{0\}$ \textit{and} $R$.

\textit{ii) For} $\widehat{\gamma }\in \mathbb{Z}_{n}-\{M_{p_{1}}\cup
M_{p_{2}}\cup ...\cup M_{p_{r}}\}$\textit{, the element} $X+\widehat{\gamma }
$ \textit{is an invertible element in} $R$ .
\end{proposition}

\begin{proof}
(i) Indeed, the ideals are: $J_{p_{i}}=\left( \widehat{\alpha }X+\widehat{%
\alpha }_{i}\right) ,\widehat{\alpha }_{i}\in M_{p\,_{i}},i\in
\{1,2,...,r\}, $ which are maximals, $J_{d}=\left( \widehat{\beta }X+%
\widehat{\beta }_{d}\right) ,\widehat{\beta }_{d}\in I_{d},I_{d}$ is not
maximal, $\widehat{\alpha },\widehat{\beta }\in R,d\neq n$, where $d$ is a
proper divisor of $n$, the ideals $\left( X\right) $, for $d=n$ and $\left(
0\right) $. Therefore, we have $\complement _{n}^{0}$ ideals, for ideal $%
\left( X\right) $, $\complement _{n}^{1}$ ideals for ideals $J_{p_{i}}$, $%
\complement _{n}^{2}$ ideals for ideals $J_{p_{i}p_{j}}$, $p_{i}\neq p_{j}$%
,...,$\complement _{n}^{n}$ ideals for ideal $R$, for $d=1$, resulting a
total of $2^{r}+1,$ if we add ideal $\left( 0\right) $.

(ii) Since $\widehat{\gamma }\in \mathbb{Z}_{n}-\{M_{p_{1}}\cup
M_{p_{2}}\cup ...\cup M_{p_{r}}\}$, we have that $\widehat{\gamma }$ is
invertible, with $\widehat{\delta }$ its inverse. Therefore, $\left( X+%
\widehat{\gamma }\right) [-\widehat{\delta }^{-2}\left( X-\widehat{\gamma }%
\right) ]=1$. It results that $X+\widehat{\gamma }$ is invertible, therefore 
$\left( X+\widehat{\gamma }\right) =R$.
\end{proof}

Since for any commutative unitary ring, its lattice of ideals is a
residuated lattice (see \cite{[TT; 22]}), in particular, for the unitary and
commutative ring \thinspace $A=\mathbb{Z}_{n}[X]/\left( X^{2}\right) $, we
have that $(Id(\mathbb{Z}_{n}/\left( X^{2}\right) ),\cap ,+,\otimes
\rightarrow ,0=\{0\},1=A)$ is a residuated lattice with $2^{r}+1$ elements.

\begin{remark}
\label{Remark7} As we remarked above, the ideals in the ring $R=\mathbb{Z}%
_{n}[X]/\left( X^{2}\right) $ are:

(i) $\left( 0\right) ;$

(ii) of the form $J_{d}=\left( \widehat{\alpha }X+\widehat{\beta }%
_{d}\right) ,\widehat{\alpha }\in R,\widehat{\beta }_{d}\in I_{d},$ where $d$
is a proper divisor of $n=p_{1}p_{2}...p_{r},p_{1},p_{2},...,p_{r}$ being
distinct prime numbers, $r\geq 2,$ by using the notations from Remark \ref%
{Remark5}. If $I_{d}=\left( \widehat{p_{i}}\right) ,$ then $J_{d}$ is
denoted $J_{p_{i}}$ and it is a maximal ideal in $R=\mathbb{Z}_{n}[X]/\left(
X^{2}\right) $ ;

(iii) The ring $R,$ if $d=1$;

(iv) $\left( X\right) ,$ if $d=n$.
\end{remark}

\begin{remark}
\label{Remark8} We remark that for all two nonzero ideals of the above ring $%
R$, $I$ and $J$, we have $\left( X\right) \subseteq I\cap J$ and the ideal $%
\left( X\right) $ is the only minimal ideal of $\mathbb{Z}_{n}[X]/\left(
X^{2}\right) $.
\end{remark}

\begin{remark}
\label{Remark9} Let $D_{d}=\{$ $p\in \{p_{1},p_{2},...,p_{r}\},$ such that $%
d=\prod p\},d\neq 1$.

(1) We have $J_{d_{1}}\cap J_{d_{2}}=J_{d_{1}}\otimes J_{d_{2}}=J_{d_{3}}$,
where $D_{d_{3}}=\{p\in D_{d_{1}}\cup D_{d_{2}},d_{3}=\prod p\}$, for $%
d_{1},d_{2}$ proper divisors.

If $d_{1}=1,$ we have $R\otimes J_{d_{2}}=J_{d_{2}}=R\cap J_{d_{2}}$.

If $d_{1}=n$, $d_{2}\neq n$, we have $\left( X\right) \otimes
J_{d_{2}}=\left( X\right) \cap J_{d_{2}}=\left( X\right) $. If $d_{2}=n,$ we
have $\left( X\right) \otimes \left( X\right) =\left( 0\right) $.

(2) We have $(J_{d_{1}}:J_{d_{2}})=J_{d_{3}}$, with $%
D_{d_{3}}=D_{d_{1}}-D_{d_{2}}$. Indeed, $(J_{d_{1}}:J_{d_{2}})=\{y\in
R,y\cdot J_{d_{2}}\subseteq J_{d_{1}}\}=J_{d_{3}}$, for $d_{1},d_{2}$ proper
divisors$.$

If $J_{d_{1}}=\left( 0\right) ,$ we have $(0:J_{d_{2}})=\left( 0\right) $.
Indeed, if $(0:J_{d_{2}})=J\neq \left( 0\right) $, it results that $J\otimes
J_{d_{2}}=\left( 0\right) $. But, from the above, $J\otimes J_{d_{2}}=$ $%
J\cap J_{d_{2}}$ $\neq \left( 0\right) $, which is false

If $J_{d_{2}}=\left( 0\right) $, we have $(J_{d_{1}}:0)=R$.

If $d_{1}=1$, we have $\left( R:J_{d_{2}}\right) =R$ and $\left(
J_{d_{2}}:R\right) =J_{d_{2}}.$

If $d_{1}=n$, $d_{2}\neq n$, we have $J_{d_{1}}=\left( X\right) $, therefore 
$(J_{d_{1}}:J_{d_{2}})=J_{d_{1}}=\left( X\right) $. If $d_{1}\neq n$, $%
d_{2}=n$, we have $J_{d_{2}}=\left( X\right) $, therefore $%
(J_{d_{1}}:J_{d_{2}})=R$. If $d_{1}=d_{2}=n$, we have $J_{d_{1}}=J_{d_{2}}=%
\left( X\right) $ and $(J_{d_{1}}:J_{d_{2}})=R$.
\end{remark}

\begin{proposition}
\label{Proposition10} (\textit{i)} \textit{For} $n\geq 2$\textit{, with the
above notations, the residuated lattice} $(Id(\mathbb{Z}_{n}[X]/\left(
X^{2}\right) ),\cap ,+,\otimes \rightarrow ,0=\{0\},1=R),R=\mathbb{Z}%
_{n}[X]/\left( X^{2}\right) ,$ \ \textit{is a BL-algebra with} $2^{r}+1$ 
\textit{elements.}

(\textit{ii)} \textit{By using notations from Remark \ref{Remark7}, we have
that} $(Id_{p_{i}}(\mathbb{Z}_{n}[X]/\left( X^{2}\right) ),\cap ,+,\otimes
\rightarrow ,0=\{0\},1=R),$\textit{where} $Id_{p_{i}}(\mathbb{Z}%
_{n}[X]/\left( X^{2}\right) )=\{\left( 0\right) ,J_{p_{i}},R\},~$\textit{is
\ a BL-sublattice of the lattice} $Id(\mathbb{Z}_{n}[X]/\left( X^{2}\right) )
$, \textit{having }$3$ \textit{elements.}
\end{proposition}

\begin{proof}
(i) First, we will prove the $\left( prel\right) $ condition: 
\begin{equation*}
\left( I\rightarrow J\right) \vee \left( J\rightarrow I\right) =(J:I)\vee
(I:J)=\mathbb{Z}_{n}[X]/\left( X^{2}\right) ,
\end{equation*}%
for every $I,J\in Id(\mathbb{Z}_{n}[X]/\left( X^{2}\right) )$.

\textit{Case 1.} If $d_{1}$ and $d_{2}$ are proper divisors of $n$, we have $%
\left( J_{d_{1}}\rightarrow J_{d_{2}}\right) \vee \left(
J_{d_{2}}\rightarrow J_{d_{1}}\right) =(J_{d_{2}}:J_{d_{1}})\vee
(J_{d_{1}}:J_{d_{2}})=J_{d_{4}}\vee J_{d_{5}},$ where $%
D_{d_{4}}=D_{d_{1}}-D_{d_{2}}$ and $D_{d_{5}}=D_{d_{2}}-D_{d_{1}}$. We
remark that $D_{d_{4}}\cap D_{d_{5}}=\emptyset ,$ then \textit{gcd}%
\thinspace $\{d_{4},d_{5}\}=1$. From here, it results that there are the
integers $a$ and $b$ such that $ad_{4}+bd_{5}=1$. We obtain that $%
J_{d_{4}}\vee J_{d_{5}}=<J_{d_{4}}\cup J_{d_{5}}>=R$, from Proposition \ref%
{Proposition6}, ii).

\textit{Case 2.} If $d_{1}$ is a proper divisor of $n$ and $d_{2}=n$, we
have $J_{d_{2}}=\left( X\right) $. Therefore, $\left( J_{d_{1}}\rightarrow
J_{d_{2}}\right) \vee \left( J_{d_{2}}\rightarrow J_{d_{1}}\right)
=(J_{d_{2}}:J_{d_{1}})\vee (J_{d_{1}}:J_{d_{2}})=J_{d_{2}}\vee R=R$, by
using Remark \ref{Remark9}.

\textit{Case 3.} If $d_{1}$ is a proper divisor of $n$ and $J_{d_{2}}=\left(
0\right) $, we have $\left( J_{d_{1}}\rightarrow J_{d_{2}}\right) \vee
\left( J_{d_{2}}\rightarrow J_{d_{1}}\right) =(0:J_{d_{1}})\vee
(J_{d_{1}}:0)=0\vee R=R$, by using Remark \ref{Remark9}.

\textit{Case 4. }If $d_{1}$ is a proper divisor of $n$ and $J_{d_{2}}=R$, it
is clear. From here, we have that the condition $\left( prel\right) $ is
satisfied.

Now, we prove condition $\left( div\right) :$ 
\begin{equation*}
I\otimes \left( I\rightarrow J\right) =I\otimes \left( J:I\right) =I\cap J,
\end{equation*}%
for every $I,J\in Id(\mathbb{Z}_{n}[X]/\left( X^{2}\right) )$.

\textit{Case 1.} If $d_{1}$ and $d_{2}$ are proper divisors of $n$, we have $%
J_{d_{1}}\otimes \left( J_{d_{2}}:J_{d_{1}}\right) =J_{d_{1}}\otimes
J_{d_{3}}=J_{d_{4}}=J_{d_{1}}\cap J_{d_{2}},$ since $%
D_{d_{3}}=D_{d_{2}}-D_{d_{1}}$ and $D_{d_{4}}=\{p\in D_{d_{1}}\cup
D_{d_{3}},d_{4}=\prod p\}=\{p\in D_{d_{1}}\cup D_{d_{2}},d_{4}=\prod p\}$.

\textit{Case 2.} If $d_{1}$ is a proper divisor of $n$ and $d_{2}=n$, we
have $J_{d_{2}}=\left( X\right) $. \ We obtain $J_{d_{1}}\otimes \left(
J_{d_{2}}:J_{d_{1}}\right) =J_{d_{1}}\otimes \left( \left( X\right)
:J_{d_{1}}\right) =J_{d_{1}}\otimes \left( X\right) =J_{d_{1}}\cap \left(
X\right) $ since $\left( X\right) \subset J_{d_{1}}$.

\textit{Case 3.} If $d_{1}=n$ and \thinspace $d_{2}$ is a proper divisor of $%
n$, we have $J_{d_{1}}=\left( X\right) $. \ We obtain $J_{d_{1}}\otimes
\left( J_{d_{2}}:J_{d_{1}}\right) =\left( X\right) \otimes \left(
J_{d_{2}}:\left( X\right) \right) =\left( X\right) \otimes R=\left( X\right)
=J_{d_{2}}\cap \left( X\right) $ since $\left( X\right) \subset J_{d_{2}}$.

\textit{Case 4.} If $d_{1}$ is a proper divisor of $n$ and $J_{d_{2}}=\left(
0\right) $, we have $J_{d_{1}}\otimes \left( J_{d_{2}}:J_{d_{1}}\right)
=J_{d_{1}}\otimes \left( 0:J_{d_{1}}\right) =J_{d_{1}}\otimes \left(
0\right) =\left( 0\right) =J_{d_{1}}\cap \left( 0\right) ,$ from Remark \ref%
{Remark9}.

\textit{Case 5.} If $J_{d_{1}}=\left( 0\right) $ and $d_{2}$ is a proper
divisor of $n$, we have $J_{d_{1}}\otimes \left( J_{d_{2}}:J_{d_{1}}\right)
=0\otimes \left( J_{d_{2}}:0\right) =0.$

\textit{Case 6. }If $d_{1}$ is a proper divisor of $n$ and $J_{d_{2}}=R$, we
have $J_{d_{1}}\otimes \left( J_{d_{2}}:J_{d_{1}}\right) =J_{d_{1}}\otimes
\left( R:J_{d_{1}}\right) =J_{d_{1}}\otimes R=J_{d_{1}}$. If $\ J_{d_{1}}=R$
and $d_{2}$ is a proper divisor of $n$, we have $J_{d_{1}}\otimes \left(
J_{d_{2}}:J_{d_{1}}\right) =R\otimes \left( J_{d_{2}}:R\right) =R\otimes
J_{d_{2}}=J_{d_{2}}$. From here, we have that the condition $\left(
div\right) $ is satisfied and the proposition is proved.

(ii) It is clear that $J_{p_{i}}\odot J_{p_{i}}=J_{p_{i}}\otimes
J_{p_{i}}=J_{p_{i}}$, and we obtain the following tables:

\begin{equation*}
\begin{tabular}{l|lll}
$\rightarrow $ & $O$ & $J_{p_{i}}$ & $R$ \\ \hline
$O$ & $R$ & $R$ & $R$ \\ 
$J_{p_{i}}$ & $O$ & $R$ & $R$ \\ 
$R$ & $O$ & $J_{p_{i}}$ & $R$%
\end{tabular}%
\ \ \ 
\begin{tabular}{l|lll}
$\odot $ & $O$ & $J_{p_{i}}$ & $R$ \\ \hline
$O$ & $O$ & $O$ & $O$ \\ 
$J_{p_{i}}$ & $O$ & $J_{p_{i}}$ & $J_{p_{i}}$ \\ 
$R$ & $O$ & $J_{p_{i}}$ & $R$%
\end{tabular}%
\ \ ,
\end{equation*}%
therefore a BL-algebra of order $3.$
\end{proof}

\begin{proposition}
\label{Proposition11} \textit{Let} $n=p^{r}$ \textit{with} $p$ \textit{a
prime number,} $p\geq 2,$ $r$ \textit{a positive integer,} $r\geq 2$. 
\textit{We consider the ring} $R=\mathbb{Z}_{n}[X]/\left( X^{2}\right) $. 
\textit{The set} $(Id(\mathbb{Z}_{n}[X]/\left( X^{2}\right) ),\cap
,+,\otimes \rightarrow ,0=\{0\},1=R)$ i\textit{s a BL-chain with} $r+2$ 
\textit{elements. In this way, for a given positive integer }$k\geq 4$, we
can construct all BL-chains with $k$ elements\textit{.}
\end{proposition}

\begin{proof}
The ideals in $\mathbb{Z}_{n}$ are of the form: $\left( 0\right) \subseteq
\left( p^{r-1}\right) \subseteq \left( p^{r-2}\right) \subseteq ...\subseteq
\left( p\right) \subseteq \mathbb{Z}_{n}$. The ideal $\left( p^{r-1}\right) $
and the ideal $\left( p\right) $ is the only maximal ideal of $\mathbb{Z}%
_{n} $. The ideals in the ring $R$ are $\left( 0\right) \subseteq \left(
X\right) \subseteq \left( \alpha _{r-1}X+\beta _{r-1}\right) \subseteq
\left( \alpha _{r-2}X+\beta _{r-2}\right) \subseteq ...\subseteq \left(
\alpha _{1}X+\beta _{1}\right) \subseteq R,$ where $\alpha _{i}\in \mathbb{Z}%
_{n},i\in \{1,...,r-1\},$ $\beta _{r-1}\in \left( p^{r-1}\right) ,\beta
_{r-2}\in \left( p^{r-2}\right) ,...,\beta _{1}\in \left( p\right) ,$
therefore $r+2$ ideals. We denote these ideals with $\left( 0\right) ,\left(
X\right) ,$ $I_{p^{r-1}},I_{p^{r-2}},...I_{p},\,R$, with $I_{p}$ the only
maximal ideal in $R$.

First, we will prove $\left( prel\right) $ condition: 
\begin{equation*}
\left( I\rightarrow J\right) \vee \left( J\rightarrow I\right) =(J:I)\vee
(I:J)=\mathbb{Z}_{n}[X]/\left( X^{2}\right) ,
\end{equation*}%
for every $I,J\in Id(\mathbb{Z}_{n}[X]/\left( X^{2}\right) )$.

\textit{Case 1.} We suppose that $I$ and $J$ are proper ideals and $%
I\subseteq J$. We have $\left( I\rightarrow J\right) \vee \left(
J\rightarrow I\right) =(J:I)\vee (I:J)=R\vee I=R$.

\textit{Case 2}$.I=\left( 0\right) $ and $J$ a proper ideal$.$ We have $%
\left( I\rightarrow J\right) \vee \left( J\rightarrow I\right) =(J:\left(
0\right) )\vee (\left( 0\right) :J)=R$. Therefore, the condition $\left(
prel\right) $ is satisfied.

Now, we prove $\left( div\right) $ condition: 
\begin{equation*}
I\otimes \left( I\rightarrow J\right) =I\otimes \left( J:I\right) =I\cap J,
\end{equation*}%
for every $I,J\in Id(\mathbb{Z}_{n}[X]/\left( X^{2}\right) )$.

\textit{Case 1.} We suppose that $I$ and $J$ are proper ideals and $%
I\subseteq J$. We have $I\otimes \left( I\rightarrow J\right) =I\otimes
\left( J:I\right) =I\otimes R=I=I\cap J$. If $J\subseteq I$, we have $%
I\otimes \left( I\rightarrow J\right) =I\otimes \left( J:I\right) =I\otimes
J=J=I\cap J$.

\textit{Case 2}$.I=\left( 0\right) $ and $J$ is a proper ideal$.$We have $%
\left( 0\right) \otimes \left( \left( 0\right) \rightarrow J\right) =\left(
0\right) \otimes \left( J:\left( 0\right) \right) =\left( 0\right) =I\cap J$%
. \ If $I\neq \left( X\right) $ is a proper ideal and $J=\left( 0\right) $,
we have $I\otimes \left( I\rightarrow J\right) =I\otimes \left( \left(
0\right) :I\right) =I\otimes \left( 0\right) =\left( 0\right) $. If $%
I=\left( X\right) $ and $J=\left( 0\right) $, we have $I\otimes \left(
I\rightarrow J\right) =\left( X\right) \otimes \left( \left( 0\right)
:\left( X\right) \right) =\left( X\right) \otimes \left( X\right) =\left(
0\right) $ and $\left( 0\right) \cap \left( X\right) =\left( 0\right) $.
From here, we have that the condition $\left( div\right) $ is satisfied and
the proposition is proved.
\end{proof}

\begin{example}
\label{Example12} In Proposition \ref{Proposition10}, we take $n=2\cdot 3\,$%
, therefore the ideals of $\ \mathbb{Z}_{6}$ are $\left( 0\right) ,\left(
2\right) ,\left( 3\right) ,\mathbb{Z}_{6}$, with $\left( 2\right) $ and $%
\left( 3\right) $ maximal ideals. The ring $\mathbb{Z}_{6}[X]/\left(
X^{2}\right) $ has $5$ ideals: $O=\left( 0\right) \subset A=\left( X\right)
,B=\left( \alpha X+\beta \right) ,C=\left( \gamma X+\delta \right) ,E=%
\mathbb{Z}_{6}$, with $\alpha ,\gamma \in $ $\mathbb{Z}_{4},\beta \in \left(
2\right) $ and $\delta \in \left( 3\right) $. From the following tables, we
have a BL-structure on $Id(\mathbb{Z}_{6}[X]/\left( X^{2}\right) )$:

\begin{equation*}
\begin{tabular}{l|lllll}
$\rightarrow $ & $O$ & $A$ & $B$ & $C$ & $E$ \\ \hline
$O$ & $E$ & $E$ & $E$ & $E$ & $E$ \\ 
$A$ & $A$ & $E$ & $E$ & $E$ & $E$ \\ 
$B$ & $O$ & $A$ & $E$ & $C$ & $E$ \\ 
$C$ & $O$ & $A$ & $B$ & $E$ & $E$ \\ 
$E$ & $O$ & $A$ & $B$ & $C$ & $E$%
\end{tabular}%
\ \ \ 
\begin{tabular}{l|lllll}
$\odot $ & $O$ & $A$ & $B$ & $C$ & $E$ \\ \hline
$O$ & $O$ & $O$ & $O$ & $O$ & $O$ \\ 
$A$ & $O$ & $O$ & $A$ & $A$ & $A$ \\ 
$B$ & $O$ & $A$ & $B$ & $A$ & $B$ \\ 
$C$ & $O$ & $A$ & $A$ & $C$ & $C$ \\ 
$E$ & $O$ & $A$ & $B$ & $C$ & $E$%
\end{tabular}%
\ \ .
\end{equation*}

From Proposition \ref{Proposition10}, if we consider $J_{p_{i}}=$ $B$, we
have the following BL-algebra of order $3$:

\begin{equation*}
\begin{tabular}{l|lll}
$\rightarrow $ & $O$ & $B$ & $E$ \\ \hline
$O$ & $E$ & $E$ & $E$ \\ 
$B$ & $O$ & $E$ & $E$ \\ 
$E$ & $O$ & $B$ & $E$%
\end{tabular}%
\ \ \ 
\begin{tabular}{l|lll}
$\odot $ & $O$ & $B$ & $E$ \\ \hline
$O$ & $O$ & $O$ & $O$ \\ 
$B$ & $O$ & $B$ & $B$ \\ 
$E$ & $O$ & $B$ & $E$%
\end{tabular}%
\ \ .
\end{equation*}
\end{example}

\begin{example}
\label{Example13} In Proposition \ref{Proposition10}, we take $n=2\cdot
3\cdot 5\,$, therefore the ideals of the ring $\mathbb{Z}_{30}$ are $\left(
0\right) ,\left( 2\right) ,\left( 3\right) ,\left( 5\right) ,\left( 6\right)
,\left( 10\right) ,\left( 15\right) ,\mathbb{Z}_{30}$, with $\left( 2\right)
,\left( 3\right) $ and $\left( 5\right) $ maximal ideals. The ring $\mathbb{Z%
}_{30}[X]/\left( X^{2}\right) $ has $9$ ideals: $O=\left( 0\right) \subset
A=\left( X\right) ,B=\left( \alpha _{1}X+\beta _{1}\right) ,C=\left( \alpha
_{2}X+\beta _{2}\right) ,$\newline
$D=\left( \alpha _{3}X+\beta _{3}\right) ,E=\left( \alpha _{4}X+\beta
_{4}\right) ,F=\left( \alpha _{5}X+\beta _{5}\right) ,G=\left( \alpha
_{6}X+\beta _{6}\right) ,R=\mathbb{Z}_{30}$, with $\alpha _{i}\in $ $\mathbb{%
Z}_{30},i\in \{1,2,3,4,5,6\},\beta _{1}\in \left( 6\right) $, $\beta _{2}\in
\left( 10\right) ,\beta _{3}\in \left( 15\right) ,\beta _{4}\in \left(
2\right) ,\beta _{5}\in \left( 3\right) $ and $\beta _{6}\in \left( 5\right) 
$. The ideals $E,F$ and $G$ are maximals. From the following tables, we have
a BL-structure on $Id(\mathbb{Z}_{30}[X]/\left( X^{2}\right) )$:%
\begin{equation*}
\begin{tabular}{l|lllllllll}
$\rightarrow $ & $O$ & $A$ & $B$ & $C$ & $D$ & $E$ & $F$ & $G$ & $R$ \\ 
\hline
$O$ & $R$ & $R$ & $R$ & $R$ & $R$ & $R$ & $R$ & $R$ & $R$ \\ 
$A$ & $A$ & $R$ & $R$ & $R$ & $R$ & $R$ & $R$ & $R$ & $R$ \\ 
$B$ & $O$ & $A$ & $R$ & $C$ & $D$ & $R$ & $R$ & $G$ & $R$ \\ 
$C$ & $O$ & $A$ & $B$ & $R$ & $D$ & $R$ & $F$ & $R$ & $R$ \\ 
$D$ & $O$ & $A$ & $B$ & $C$ & $R$ & $E$ & $R$ & $R$ & $R$ \\ 
$E$ & $O$ & $A$ & $B$ & $C$ & $D$ & $R$ & $F$ & $G$ & $R$ \\ 
$F$ & $O$ & $A$ & $B$ & $C$ & $D$ & $E$ & $R$ & $G$ & $R$ \\ 
$G$ & $O$ & $A$ & $B$ & $C$ & $D$ & $E$ & $F$ & $R$ & $R$ \\ 
$R$ & $O$ & $A$ & $B$ & $C$ & $D$ & $E$ & $F$ & $G$ & $R$%
\end{tabular}%
\ \ 
\begin{tabular}{l|lllllllll}
$\odot $ & $O$ & $A$ & $B$ & $C$ & $D$ & $E$ & $F$ & $G$ & $R$ \\ \hline
$O$ & $O$ & $O$ & $O$ & $O$ & $O$ & $O$ & $O$ & $O$ & $O$ \\ 
$A$ & $O$ & $O$ & $A$ & $A$ & $A$ & $A$ & $A$ & $A$ & $A$ \\ 
$B$ & $O$ & $A$ & $B$ & $A$ & $A$ & $B$ & $B$ & $A$ & $B$ \\ 
$C$ & $O$ & $A$ & $A$ & $C$ & $A$ & $C$ & $C$ & $A$ & $C$ \\ 
$D$ & $O$ & $A$ & $A$ & $A$ & $D$ & $A$ & $D$ & $D$ & $D$ \\ 
$E$ & $O$ & $A$ & $B$ & $C$ & $A$ & $E$ & $A$ & $A$ & $E$ \\ 
$F$ & $O$ & $A$ & $B$ & $A$ & $D$ & $A$ & $F$ & $A$ & $F$ \\ 
$G$ & $O$ & $A$ & $A$ & $C$ & $D$ & $A$ & $A$ & $G$ & $G$ \\ 
$R$ & $O$ & $A$ & $B$ & $C$ & $D$ & $E$ & $F$ & $G$ & $R$%
\end{tabular}%
\text{.}
\end{equation*}
\end{example}

\begin{example}
\label{Example14} In Proposition \ref{Proposition11}, we consider $p=2,r=2$.
The ideals in $\left( \mathbb{Z}_{4},+,\cdot \right) $ are $\left( 0\right)
\subset \left( 2\right) \subset \mathbb{Z}_{4}$ and $\mathbb{Z}_{4}$ is a
local ring. The ring $\mathbb{Z}_{4}[X]/\left( X^{2}\right) $ has $4$
ideals: $O=\left( 0\right) \subset A=\left( X\right) \subset B=\left( \alpha
X+\beta \right) \subset E=\mathbb{Z}_{4}$, with $\alpha ,\gamma \in $ $%
\mathbb{Z}_{4},\beta \in \left( 2\right) $. From the following tables, we
have a BL-structure for $Id(\mathbb{Z}_{4}[X]/\left( X^{2}\right) )$:

\begin{equation*}
\begin{tabular}{l|llll}
$\rightarrow $ & $O$ & $A$ & $B$ & $E$ \\ \hline
$O$ & $E$ & $E$ & $E$ & $E$ \\ 
$A$ & $A$ & $E$ & $E$ & $E$ \\ 
$B$ & $O$ & $A$ & $E$ & $E$ \\ 
$E$ & $O$ & $A$ & $B$ & $E$%
\end{tabular}%
\ \ \ 
\begin{tabular}{l|llll}
$\odot $ & $O$ & $A$ & $B$ & $E$ \\ \hline
$O$ & $O$ & $O$ & $O$ & $O$ \\ 
$A$ & $O$ & $O$ & $A$ & $A$ \\ 
$B$ & $O$ & $A$ & $A$ & $B$ \\ 
$E$ & $O$ & $A$ & $B$ & $E$%
\end{tabular}%
\ \ .
\end{equation*}
\end{example}

\begin{example}
\label{Example15} In Proposition \ref{Proposition11}, we consider $p=2,r=3$.
The ideals in $\left( \mathbb{Z}_{8},+,\cdot \right) $ are $\left( 0\right)
\subset \left( 4\right) \subset \left( 2\right) \subset \mathbb{Z}_{8}$. The
ring $\mathbb{Z}_{8}[X]/\left( X^{2}\right) $ has $5$ ideals: $O=\left(
0\right) \subset A=\left( X\right) \subset B=\left( \alpha X+\beta \right)
\subset C=\left( \gamma X+\delta \right) \subset E=\mathbb{Z}_{8}$, with $%
\alpha ,\gamma \in $ $\mathbb{Z}_{8},\beta \in \left( 4\right) $ and $\delta
\in \left( 2\right) $. From the following tables, we have a BL-structure for 
$Id(\mathbb{Z}_{8}[X]/\left( X^{2}\right) )$:
\end{example}

\begin{equation*}
\begin{tabular}{l|lllll}
$\rightarrow $ & $O$ & $A$ & $B$ & $C$ & $E$ \\ \hline
$O$ & $E$ & $E$ & $E$ & $E$ & $E$ \\ 
$A$ & $A$ & $E$ & $E$ & $E$ & $E$ \\ 
$B$ & $O$ & $A$ & $E$ & $E$ & $E$ \\ 
$C$ & $O$ & $A$ & $B$ & $E$ & $E$ \\ 
$E$ & $O$ & $A$ & $B$ & $C$ & $E$%
\end{tabular}%
\ \ \ 
\begin{tabular}{l|lllll}
$\odot $ & $O$ & $A$ & $B$ & $C$ & $E$ \\ \hline
$O$ & $O$ & $O$ & $O$ & $O$ & $O$ \\ 
$A$ & $O$ & $O$ & $A$ & $A$ & $A$ \\ 
$B$ & $O$ & $A$ & $A$ & $A$ & $B$ \\ 
$C$ & $O$ & $A$ & $A$ & $B$ & $C$ \\ 
$E$ & $O$ & $A$ & $B$ & $C$ & $E$%
\end{tabular}%
\ \ .
\end{equation*}%
In the following, we present a way to generate new finite\ BL-algebras using
the ordinal product of residuated lattices.

We recall that, in \cite{[I; 09]}, Iorgulescu study the influence of the
conditions $(prel)$ and $(div)$ on the ordinal product of two residuated
lattices.

It is known that, if $\mathcal{L}_{1}=(L_{1},\wedge _{1},\vee _{1},\odot
_{1},\rightarrow _{1},0_{1},1_{1})$ and $\mathcal{L}_{2}=(L_{2},\wedge
_{2},\vee _{2},\odot _{2},\rightarrow _{2},0_{2},1_{2})$ are two residuated
lattices such that $1_{1}=0_{2}$ and $(M_{1}\backslash \{1_{1}\})\cap
(M_{2}\backslash \{0_{2}\})=\oslash ,$ then, the ordinal product of $%
\mathcal{L}_{1}$ and $\mathcal{L}_{2}$ is the residuated lattice $\mathcal{L}%
_{1}\boxtimes \mathcal{L}_{2}=(L_{1}\cup L_{2},\wedge ,\vee ,\odot
,\rightarrow ,0,1)$ where

\begin{equation*}
0=0_{1}\text{ and }1=1_{2},
\end{equation*}%
\begin{equation*}
x\leq y\text{ if }(x,y\in L_{1}\text{ and }x\leq _{1}y)\text{ or }(x,y\in
L_{2}\text{ and }x\leq _{2}y)\text{ or }(x\in L_{1}\text{ and }y\in L_{2})%
\text{ ,}
\end{equation*}

\begin{equation*}
x\wedge y=\left\{ 
\begin{array}{c}
x\wedge _{1}y,\text{ if }x,y\in L_{1}, \\ 
x\wedge _{2}y,\text{ if }x,y\in L_{2}, \\ 
x,\text{ if }x\in L_{1}\text{ and }y\in L_{2}%
\end{array}%
\right.
\end{equation*}%
\begin{equation*}
x\vee y=\left\{ 
\begin{array}{c}
x\vee _{1}y,\text{ if }x,y\in L_{1}, \\ 
x\vee _{2}y,\text{ if }x,y\in L_{2}, \\ 
y,\text{ if }x\in L_{1}\text{ and }y\in L_{2}%
\end{array}%
\right.
\end{equation*}%
\begin{equation*}
x\rightarrow y=\left\{ 
\begin{array}{c}
1,\text{ if }x\leq y, \\ 
x\rightarrow _{i}y,\text{ if }x\nleq y,\text{ }x,y\in L_{i},\text{ }i=1,2,
\\ 
y,\text{ if }x\nleq y,\text{ }x\in L_{2},\text{ }y\in L_{1}\backslash
\{1_{1}\}.%
\end{array}%
\right.
\end{equation*}%
\begin{equation*}
x\odot y=\left\{ 
\begin{array}{c}
x\odot _{1}y,\text{ if }x,y\in L_{1}, \\ 
x\odot _{2}y,\text{ if }x,y\in L_{2}, \\ 
x,\text{ if }x\in L_{1}\backslash \{1_{1}\}\text{ and }y\in L_{2}.%
\end{array}%
\right.
\end{equation*}

\bigskip The ordinal product is associative, but is not commutative, see 
\cite{[I; 09]}.

\begin{proposition}
\label{Proposition_5} ( \cite{[I; 09]}, Corollary 3.5.10) Let $\mathcal{L}%
_{1}$ and $\mathcal{L}_{2}$ be BL-algebras.

\textit{(i) }\ If $\mathcal{L}_{1}$ is a chain, then the ordinal product $%
\mathcal{L}_{1}\boxtimes \mathcal{L}_{2}$ \textit{\ is a BL-algebra.}

\textit{(ii) }$\ \ $If $\mathcal{L}_{1}$ is not a chain, then the ordinal
product $\mathcal{L}_{1}\boxtimes \mathcal{L}_{2}$ \textit{\ is only a
residuated lattice satisfying (div) condition.}
\end{proposition}

\begin{remark}
\label{Remark_3} (i) An ordinal product of two BL-chains is a BL-chain.
Indeed, using the definition of implication in an ordinal product for every $%
x,y$ we have $x\rightarrow y=1$ or $y\rightarrow x=1.$

(ii) An ordinal product of two MV-algebras is a BL-algebra which is not an
MV-algebra. Indeed, if $\mathcal{L}_{1}$ and $\mathcal{L}_{2}$ are two
MV-algebras, using Proposition \ref{Proposition_5}, the residuated lattice $%
\mathcal{L}_{1}\boxtimes \mathcal{L}_{2}$ \textit{\ is a BL-algebra in
which, we have} $(1_{1})^{\ast \ast }=(1_{1}\rightarrow 0_{1})^{\ast
}=(0_{1})^{\ast }=0_{1}\rightarrow 0_{1}=1=1_{2}\neq 1_{1},$ so, $\mathcal{L}%
_{1}\boxtimes \mathcal{L}_{2}$ \textit{\ is not an MV-algebra.}
\end{remark}

For a natural number $n\geq 2$ we consider the decomposition (which is not
unique) of $n$ in factors greater than $1$ and we denote by $\pi (n)$ the
number of all such decompositions. Obviously, if $n$ is prime then $\pi
(n)=0.$

We recall that an MV-algebra is finite iff it is isomorphic to a finite
product of MV-chains, see \cite{[HR; 99]}. \ Using this result, in \cite%
{[FP; 22]} we showed that for every natural number $n\geq 2$ there are $\pi
(n)+1$ non-isomorphic MV-algebras with $n$ elements of which only one is a
chain. Moreover, all finite MV-algebras (up to an isomorphisms) with $n$
elememts correspond to finite commutative ring $A$ in which $\left\vert
Id(A)\right\vert =n.$

In \textbf{Table 1} we sum briefly describe a way to generate finite\
MV-algebras $M$ with $2\leq n\leq 6$ elements using commutative rings, see 
\cite{[FP; 22]}.

\begin{equation*}
\text{\textbf{Table 1:}}
\end{equation*}

\textbf{\ }%
\begin{tabular}{lll}
$\left\vert M\right\vert \mathbf{=n}$ & \textbf{Nr of MV} & \textbf{Rings
which generates MV} \\ 
$n=2$ & $1$ & $Id(\mathbb{Z}_{2})$ (chain) \\ 
$n=3$ & $1$ & $Id(\mathbb{Z}_{4})$ (chain) \\ 
$n=4$ & $2$ & $Id(\mathbb{Z}_{8})$ (chain) and $Id(\mathbb{Z}_{2}\times 
\mathbb{Z}_{2})$ \\ 
$n=5$ & $1$ & $Id(\mathbb{Z}_{16})$ (chain) \\ 
$n=6$ & $2$ & $Id(\mathbb{Z}_{32})$ (chain) and $Id(\mathbb{Z}_{2}\times 
\mathbb{Z}_{4})$ \\ 
$n=7$ & $1$ & $Id(\mathbb{Z}_{64})$ (chain) \\ 
$n=8$ & $3$ & $Id(\mathbb{Z}_{128})$ (chain) and $Id(\mathbb{Z}_{2}\times 
\mathbb{Z}_{8})$ and $Id(\mathbb{Z}_{2}\times \mathbb{Z}_{2}\times \mathbb{Z}%
_{2})$%
\end{tabular}

Using the construction of ordinal product, Proposition \ref{Proposition_5}
and Remark \ref{Remark_3}, we can generate BL-algebras (which are not
MV-algebras) using commutative rings.

\begin{example}
\label{Example_5} In \cite{[FP; 22]} we show that there is one MV-algebra
with 3 elements (up to an isomorphism), see Table 1.\ This MV-algebras is
isomorphic to $Id(\mathbb{Z}_{4})$ and is a chain.\ To generate a BL-chain
with 3 elements (which is not an MV-algebra) using the ordinal product we
must consider only the MV-algebra with 2 elements (which is in fact a
Boolean algebra). In \ the commutative ring $\ (\mathbb{Z}_{2},+,\cdot )$
the ideals are $Id(\mathbb{Z}_{2})=\{\{\widehat{0}\},$ $\mathbb{Z}_{2}\}$.
Using Proposition \ref{Proposition1}, $(Id\left( \mathbb{Z}_{2}\right) ,\cap
,+,\otimes \rightarrow ,0=\{0\},1=\mathbb{Z}_{2})$ is an MV-chain with the
following operations: 
\begin{equation*}
\text{ }%
\begin{tabular}{l|ll}
$\rightarrow $ & $\{\widehat{0}\}$ & $\mathbb{Z}_{2}$ \\ \hline
$\{\widehat{0}\}$ & $\mathbb{Z}_{2}$ & $\mathbb{Z}_{2}$ \\ 
$\mathbb{Z}_{2}$ & $\{\widehat{0}\}$ & $\mathbb{Z}_{2}$%
\end{tabular}%
;\text{ }%
\begin{tabular}{l|ll}
$\otimes =\cap $ & $\{\widehat{0}\}$ & $\mathbb{Z}_{2}$ \\ \hline
$\{\widehat{0}\}$ & $\{\widehat{0}\}$ & $\{\widehat{0}\}$ \\ 
$\mathbb{Z}_{2}$ & $\{\widehat{0}\}$ & $\mathbb{Z}_{2}$%
\end{tabular}%
\text{ and }%
\begin{tabular}{l|ll}
$+$ & $\{\widehat{0}\}$ & $\mathbb{Z}_{2}$ \\ \hline
$\{\widehat{0}\}$ & $\{\widehat{0}\}$ & $\mathbb{Z}_{2}$ \\ 
$\mathbb{Z}_{2}$ & $\mathbb{Z}_{2}$ & $\mathbb{Z}_{2}$%
\end{tabular}%
.
\end{equation*}%
Now we consider two MV-algebras isomorphic with $Id\left( \mathbb{Z}%
_{2}\right) $ denoted $\mathcal{L}_{1}=(L_{1}=\{0,a\},\wedge ,\vee ,\odot
,\rightarrow ,0,a)$ and $\mathcal{L}_{2}=(L_{2}=\{a,1\},\wedge ,\vee ,\odot
,\rightarrow ,a,1).$ Using Proposition \ref{Proposition_5} we can construct
the BL-algebra $\mathcal{L}_{1}\boxtimes \mathcal{L}_{2}=(L_{1}\cup
L_{2}=\{0,a.1\},\wedge ,\vee ,\odot ,\rightarrow ,0,1)$ with $0\leq a\leq 1$
and the following operations:%
\begin{equation*}
\begin{tabular}{l|lll}
$\rightarrow $ & $0$ & $a$ & $1$ \\ \hline
$0$ & $1$ & $1$ & $1$ \\ 
$a$ & $0$ & $1$ & $1$ \\ 
$1$ & $0$ & $a$ & $1$%
\end{tabular}%
\text{ and }%
\begin{tabular}{l|lll}
$\odot $ & $0$ & $a$ & $1$ \\ \hline
$0$ & $0$ & $0$ & $0$ \\ 
$a$ & $0$ & $a$ & $a$ \\ 
$1$ & $0$ & $a$ & $1$%
\end{tabular}%
\text{ ,}
\end{equation*}%
obtaining the same BL-algebra of order $3$ as in Example \ref{Example12}.
\end{example}

\begin{example}
\label{Example_6} To generate the non-linearly ordered BL-algebra with 5
elements from Example \ref{Example_3}, we consider the commutative rings $\ (%
\mathbb{Z}_{2},+,\cdot )$ and $(\mathbb{Z}_{2}\times \mathbb{Z}_{2},+,\cdot
).$ For $\mathbb{Z}_{2}\times \mathbb{Z}_{2}=\{\left( \widehat{0},\widehat{0}%
\right) ,\left( \widehat{0},\widehat{1}\right) ,\left( \widehat{1},\widehat{0%
}\right) ,\left( \widehat{1},\widehat{1}\right) \}$, we obtain the lattice $%
Id\left( \mathbb{Z}_{2}\times \mathbb{Z}_{2}\right) =\{\left( \widehat{0},%
\widehat{0}\right) ,\{\left( \widehat{0},\widehat{0}\right) ,\left( \widehat{%
0},\widehat{1}\right) \},\{\left( \widehat{0},\widehat{0}\right) ,\left( 
\widehat{1},\widehat{0}\right) \},\mathbb{Z}_{2}\times \mathbb{Z}%
_{2}\}=\{O,R,B,E\}$, which is an MV-algebra $(Id\left( \mathbb{Z}_{2}\times 
\mathbb{Z}_{2}\right) ,\cap ,+,\otimes \rightarrow ,0=\{\left( \widehat{0},%
\widehat{0}\right) \},1=\mathbb{Z}_{2}\times \mathbb{Z}_{2}),$ see
Proposition \ref{Proposition1}. In $Id\left( \mathbb{Z}_{2}\times \mathbb{Z}%
_{2}\right) $ we have the following operations: 
\begin{equation*}
\begin{tabular}{l|llll}
$\rightarrow $ & $O$ & $R$ & $B$ & $E$ \\ \hline
$O$ & $E$ & $E$ & $E$ & $E$ \\ 
$R$ & $B$ & $E$ & $B$ & $E$ \\ 
$B$ & $R$ & $R$ & $E$ & $E$ \\ 
$E$ & $O$ & $R$ & $B$ & $E$%
\end{tabular}%
,\text{ }%
\begin{tabular}{l|llll}
$\otimes =\cap $ & $O$ & $R$ & $B$ & $E$ \\ \hline
$O$ & $O$ & $O$ & $O$ & $O$ \\ 
$R$ & $O$ & $R$ & $O$ & $R$ \\ 
$B$ & $O$ & $O$ & $B$ & $B$ \\ 
$E$ & $O$ & $R$ & $B$ & $E$%
\end{tabular}%
\text{ and }%
\begin{tabular}{l|llll}
$+$ & $O$ & $R$ & $B$ & $E$ \\ \hline
$O$ & $O$ & $R$ & $B$ & $E$ \\ 
$R$ & $R$ & $R$ & $E$ & $E$ \\ 
$B$ & $B$ & $E$ & $B$ & $E$ \\ 
$E$ & $E$ & $E$ & $E$ & $E$%
\end{tabular}%
.
\end{equation*}%
If we consider two MV-algebras isomorphic with $(Id\left( \mathbb{Z}%
_{2}\right) ,\cap ,+,\otimes \rightarrow ,0=\{0\},1=\mathbb{Z}_{2})$ and $%
(Id\left( \mathbb{Z}_{2}\times \mathbb{Z}_{2}\right) ,\cap ,+,\otimes
\rightarrow ,0=\{\left( \widehat{0},\widehat{0}\right) \},1=\mathbb{Z}%
_{2}\times \mathbb{Z}_{2}),$ denoted by $\mathcal{L}_{1}=(L_{1}=\{0,c\},%
\wedge _{1},\vee _{1},\odot _{1},\rightarrow _{1},0,c)$ and $\mathcal{L}%
_{2}=(L_{2}=\{c,a,b,1\},\wedge _{2},\vee _{2},\odot _{2},\rightarrow
_{2},c,1),$ using Proposition \ref{Proposition_5} we generate the BL-algebra 
$\mathcal{L}_{1}\boxtimes \mathcal{L}_{2}=(L_{1}\cup
L_{2}=\{0,c,a,b,1\},\wedge ,\vee ,\odot ,\rightarrow ,0,1)$ from Example \ref%
{Example_3}.
\end{example}

From Proposition \ref{Proposition_5} and Remark \ref{Remark_3}, we deduce
that:

\begin{proposition}
\label{Remark_4} (i) To generate a BL-algebra with $n$ elements as an
ordinal product $\mathcal{L}_{1}\boxtimes \mathcal{L}_{2}$ of $\ $two
Bl-algebras $\mathcal{L}_{1}$ and $\mathcal{L}_{2}$ we have the following
possibilities:%
\begin{equation*}
\mathcal{L}_{1}\text{ is a BL-chain with }i\text{ elements and }\mathcal{L}%
_{2}\text{ is a BL-algebra with }j\text{ elements, }
\end{equation*}%
and%
\begin{equation*}
\mathcal{L}_{1}\text{ is a BL-chain with }j\text{ elements and }\mathcal{L}%
_{2}\text{ is a BL-algebra with }i\text{ elements, }
\end{equation*}%
or%
\begin{equation*}
\mathcal{L}_{1}\text{ is a BL-chain with }k\text{ elements and }\mathcal{L}%
_{2}\text{ is a BL-algebra with }k\text{ elements, }
\end{equation*}%
for $i,j\geq 2,i+j=n+1,$ $i<j$ and $k\geq 2,$ $k=\frac{n+1}{2}$ $\in N,$

(ii) To generate a BL-chain with $n$ elements as the ordinal product $%
\mathcal{L}_{1}\boxtimes \mathcal{L}_{2}$ of $\ $two Bl-algebras $\mathcal{L}%
_{1}$ and $\mathcal{L}_{2}$ we have the following possibilities:%
\begin{equation*}
\mathcal{L}_{1}\text{ is a BL-chain with }i\text{ elements and }\mathcal{L}%
_{2}\text{ is a BL-chain with }j\text{ elements, }
\end{equation*}%
and%
\begin{equation*}
\mathcal{L}_{1}\text{ is a BL-chain with }j\text{ elements and }\mathcal{L}%
_{2}\text{ is a BL-chain with }i\text{ elements, }
\end{equation*}%
or%
\begin{equation*}
\mathcal{L}_{1}\text{ is a BL-chain with }k\text{ elements and }\mathcal{L}%
_{2}\text{ is a BL-chain with }k\text{ elements, }
\end{equation*}%
for $i,j\geq 2,i+j=n+1,$ $i<j$ and $k\geq 2,$ $k=\frac{n+1}{2}$ $\in N.$
\end{proposition}

We make the following notations: 
\begin{equation*}
\mathcal{BL}_{n}=\text{the set of BL-algebras with }n\text{ elements}
\end{equation*}%
\begin{equation*}
\mathcal{BL}_{n}(c)=\text{the set of BL-chains with }n\text{ elements}
\end{equation*}%
\begin{equation*}
\mathcal{MV}_{n}=\text{the set of MV-algebras with }n\text{ elements}
\end{equation*}%
\begin{equation*}
\mathcal{MV}_{n}(c)=\text{the set of MV-algebras with }n\text{ elements.}
\end{equation*}

\begin{theorem}
\label{Theorem_4} (i) All finite Bl-algebras (up to an isomorphism) which
are not MV-algebras with $2\leq n\leq 5$ elements can be generated using the
ordinal product of BL-algebras.

(ii) The number of non-isomorphic BL-algebras with $n$ elements (with $2\leq
n\leq 5$) is%
\begin{equation*}
\left\vert \mathcal{BL}_{2}\right\vert =\left\vert \mathcal{MV}%
_{2}\right\vert =\pi (2)+1,
\end{equation*}%
\begin{equation*}
\left\vert \mathcal{BL}_{3}\right\vert =\left\vert \mathcal{MV}%
_{3}\right\vert +\left\vert \mathcal{BL}_{2}\right\vert =\pi (3)+\pi (2)+2,
\end{equation*}%
\begin{equation*}
\left\vert \mathcal{BL}_{4}\right\vert =\left\vert \mathcal{MV}%
_{4}\right\vert +\left\vert \mathcal{BL}_{3}\right\vert +\left\vert \mathcal{%
BL}_{2}\right\vert =\pi (4)+\pi (3)+2\cdot \pi (2)+4,
\end{equation*}%
\begin{equation*}
\left\vert \mathcal{BL}_{5}\right\vert =\left\vert \mathcal{MV}%
_{5}\right\vert +\left\vert \mathcal{BL}_{4}\right\vert +\left\vert \mathcal{%
BL}_{3}\right\vert +\left\vert \mathcal{BL}_{2}\right\vert =
\end{equation*}%
\begin{equation*}
=\pi (5)+\pi (4)+2\cdot \pi (3)+4\cdot \pi (2)+8.
\end{equation*}
\end{theorem}

\begin{proof}
From Proposition \ref{Proposition_5} and Remark \ref{Remark_3}, we remark
that using the ordinal product of two BL-algebras we can generate only
BL-algebras which are not MV-algebras.

We generate all BL-algebras with $n$ elements ($2\leq n\leq 5)$ which are
not MV-algebras.

\textbf{Case }$n=2.$

We have obviously an only BL-algebra (up to an isomorphism) isomorphic with 
\begin{equation*}
(Id(\mathbb{Z}_{2}),\cap ,+,\otimes \rightarrow ,0=\{0\},1=\mathbb{Z}_{2}).
\end{equation*}%
In fact, this residuated lattice is a BL-chain and is the only MV-algebra
with $2$ elements. We deduce that 
\begin{equation*}
\left\vert \mathcal{MV}_{2}\right\vert =\left\vert \mathcal{BL}%
_{2}\right\vert =\pi (2)+1=1
\end{equation*}%
\begin{equation*}
\left\vert \mathcal{MV}_{2}(c)\right\vert =\left\vert \mathcal{BL}%
_{2}(c)\right\vert =1.
\end{equation*}%
\textbf{Case }$n=3.$

Using Proposition \ref{Remark_4}, to generate a BL-algebra with $3$ elements
as an ordinal product $\mathcal{L}_{1}\boxtimes \mathcal{L}_{2}$ of $\ $two
BL-algebras $\mathcal{L}_{1}$ and $\mathcal{L}_{2}$ we must consider:%
\begin{equation*}
\mathcal{L}_{1}\text{ a BL-chain with }2\text{ elements and }\mathcal{L}_{2}%
\text{ a BL-algebra with }2\text{ elements. }
\end{equation*}%
Since there is only one BL-algebra with $2$ elements and it is a chain, we
obtain the BL-algebra 
\begin{equation*}
Id(\mathbb{Z}_{2})\boxtimes Id(\mathbb{Z}_{2}),
\end{equation*}%
which is a chain.

We deduce that 
\begin{equation*}
\left\vert \mathcal{MV}_{3}\right\vert =\pi (3)+1\text{ and }\left\vert 
\mathcal{BL}_{3}\right\vert =\left\vert \mathcal{MV}_{3}\right\vert +1\cdot
\left\vert \mathcal{BL}_{2}\right\vert =\pi (3)+\pi (2)+2=2
\end{equation*}%
\begin{equation*}
\left\vert \mathcal{MV}_{3}(c)\right\vert =1\text{ and }\left\vert \mathcal{%
BL}_{3}(c)\right\vert =\left\vert \mathcal{MV}_{3}(c)\right\vert +1=1+1=2.
\end{equation*}%
We remark that $\left\vert \mathcal{BL}_{3}\right\vert =\left\vert \mathcal{%
MV}_{3}\right\vert +\left\vert \mathcal{BL}_{2}\right\vert .$

\textbf{Case }$n=4.$

Using Proposition \ref{Remark_4}, to generate a BL-algebra with $4$ elements
as the ordinal product $\mathcal{L}_{1}\boxtimes \mathcal{L}_{2}$ of \ two
Bl-algebras $\mathcal{L}_{1}$ and $\mathcal{L}_{2}$ we must consider:%
\begin{equation*}
\mathcal{L}_{1}\text{ is a BL-chain with }2\text{ elements and }\mathcal{L}%
_{2}\text{ is a BL-algebra with }3\text{ elements, }
\end{equation*}%
\begin{equation*}
\mathcal{L}_{1}\text{ is a BL-chain with }3\text{ elements and }\mathcal{L}%
_{2}\text{ is a BL-algebra with }2\text{ elements. }
\end{equation*}%
We obtain the following BL-algebras: 
\begin{equation*}
Id(\mathbb{Z}_{2})\boxtimes Id(\mathbb{Z}_{4})\text{ and }Id(\mathbb{Z}%
_{2})\boxtimes (Id(\mathbb{Z}_{2})\boxtimes Id(\mathbb{Z}_{2}))\text{ }
\end{equation*}%
and 
\begin{equation*}
Id(\mathbb{Z}_{4})\boxtimes Id(\mathbb{Z}_{2})\text{ and }(Id(\mathbb{Z}%
_{2})\boxtimes (Id(\mathbb{Z}_{2}))\boxtimes Id(\mathbb{Z}_{2}).
\end{equation*}

Since $\boxtimes $ is associative, we obtain 3 BL-algebras which are chains,
from Remark \ref{Remark_3}.

We deduce that%
\begin{equation*}
\left\vert \mathcal{MV}_{4}\right\vert =\pi (4)+1\text{ }
\end{equation*}
\begin{equation*}
\left\vert \mathcal{BL}_{4}\right\vert =\left\vert \mathcal{MV}%
_{4}\right\vert +1\cdot \left\vert \mathcal{BL}_{3}\right\vert +2\cdot
\left\vert \mathcal{BL}_{2}\right\vert -1=\pi (4)+\pi (3)+2\cdot \pi (2)+4=5
\end{equation*}%
\begin{equation*}
\left\vert \mathcal{MV}_{4}(c)\right\vert =1\text{ and }\left\vert \mathcal{%
BL}_{4}(c)\right\vert =\left\vert \mathcal{MV}_{3}(c)\right\vert +3=1+3=4.
\end{equation*}

We remark that $\left\vert \mathcal{BL}_{4}\right\vert =\left\vert \mathcal{%
MV}_{4}\right\vert +\left\vert \mathcal{BL}_{3}\right\vert +\left\vert 
\mathcal{BL}_{2}\right\vert .$

\textbf{Case }$n=5.$

To generate a BL-algebra with $5$ elements as the ordinal product $\mathcal{L%
}_{1}\boxtimes \mathcal{L}_{2}$ of $\ $two Bl-algebras $\mathcal{L}_{1}$ and 
$\mathcal{L}_{2}$ we must consider:%
\begin{equation*}
\mathcal{L}_{1}\text{ is a BL-chain with }2\text{ elements and }\mathcal{L}%
_{2}\text{ is a BL-algebra with }4\text{ elements, }
\end{equation*}%
\begin{equation*}
\mathcal{L}_{1}\text{ is a BL-chain with }4\text{ elements and }\mathcal{L}%
_{2}\text{ is a BL-algebra with }2\text{ elements. }
\end{equation*}%
\begin{equation*}
\mathcal{L}_{1}\text{ is a BL-chain with }3\text{ elements and }\mathcal{L}%
_{2}\text{ is a BL-algebra with }3\text{ elements,.}
\end{equation*}%
We obtain the following BL-algebras: 
\begin{eqnarray*}
&&Id(\mathbb{Z}_{2})\boxtimes Id(\mathbb{Z}_{8}),\text{ }Id(\mathbb{Z}%
_{2})\boxtimes Id(\mathbb{Z}_{2}\times \mathbb{Z}_{2}),\text{ }Id(\mathbb{Z}%
_{2})\boxtimes (Id(\mathbb{Z}_{2})\boxtimes Id(\mathbb{Z}_{4})), \\
&&Id(\mathbb{Z}_{2})\boxtimes \lbrack Id(\mathbb{Z}_{4})\boxtimes Id(\mathbb{%
Z}_{2})]\text{ and }Id(\mathbb{Z}_{2})\boxtimes \lbrack Id(\mathbb{Z}%
_{2})\boxtimes (Id(\mathbb{Z}_{2})\boxtimes Id(\mathbb{Z}_{2}))]
\end{eqnarray*}

and 
\begin{eqnarray*}
&&Id(\mathbb{Z}_{8})\boxtimes Id(\mathbb{Z}_{2}),[Id(\mathbb{Z}%
_{2})\boxtimes Id(\mathbb{Z}_{4})]\boxtimes Id(\mathbb{Z}_{2}), \\
&&[Id(\mathbb{Z}_{4})\boxtimes Id(\mathbb{Z}_{2})]\boxtimes Id(\mathbb{Z}%
_{2})\text{ and }[Id(\mathbb{Z}_{2})\boxtimes (Id(\mathbb{Z}_{2})\boxtimes
Id(\mathbb{Z}_{2}))]\boxtimes Id(\mathbb{Z}_{2})
\end{eqnarray*}

and 
\begin{eqnarray*}
&&Id(\mathbb{Z}_{4})\boxtimes Id(\mathbb{Z}_{4}),\text{ }[Id(\mathbb{Z}%
_{2})\boxtimes Id(\mathbb{Z}_{2})]\boxtimes \lbrack Id(\mathbb{Z}%
_{2}))\boxtimes Id(\mathbb{Z}_{2})] \\
&&Id(\mathbb{Z}_{4})\boxtimes \lbrack Id(\mathbb{Z}_{2})\boxtimes Id(\mathbb{%
Z}_{2})]\text{ and }[Id(\mathbb{Z}_{2})\boxtimes Id(\mathbb{Z}%
_{2})]\boxtimes Id(\mathbb{Z}_{4})
\end{eqnarray*}

Since $\boxtimes $ is associative, $Id(\mathbb{Z}_{2})\boxtimes \lbrack Id(%
\mathbb{Z}_{4})\boxtimes Id(\mathbb{Z}_{2})]=[Id(\mathbb{Z}_{2})\boxtimes Id(%
\mathbb{Z}_{4})]\boxtimes Id(\mathbb{Z}_{2}),$ $Id(\mathbb{Z}_{2})\boxtimes
\lbrack Id(\mathbb{Z}_{2})\boxtimes (Id(\mathbb{Z}_{2})\boxtimes Id(\mathbb{Z%
}_{2}))]$ $=$ $[Id(\mathbb{Z}_{2})\boxtimes (Id(\mathbb{Z}_{2})\boxtimes Id(%
\mathbb{Z}_{2}))]\boxtimes Id(\mathbb{Z}_{2})$ $=[Id(\mathbb{Z}%
_{2})\boxtimes Id(\mathbb{Z}_{2})]\boxtimes \lbrack Id(\mathbb{Z}%
_{2}))\boxtimes Id(\mathbb{Z}_{2})],$ \ $[Id(\mathbb{Z}_{4})\boxtimes Id(%
\mathbb{Z}_{2})]\boxtimes Id(\mathbb{Z}_{2})=Id(\mathbb{Z}_{4})\boxtimes
\lbrack Id(\mathbb{Z}_{2})\boxtimes Id(\mathbb{Z}_{2})],$ $Id(\mathbb{Z}%
_{2})\boxtimes (Id(\mathbb{Z}_{2})\boxtimes Id(\mathbb{Z}_{4}))=$ $[Id(%
\mathbb{Z}_{2})\boxtimes Id(\mathbb{Z}_{2})]\boxtimes Id(\mathbb{Z}_{4})$
and $Id(\mathbb{Z}_{2})\boxtimes (Id(\mathbb{Z}_{2})\boxtimes Id(\mathbb{Z}%
_{4}))=[Id(\mathbb{Z}_{2})\boxtimes Id(\mathbb{Z}_{2})]\boxtimes Id(\mathbb{Z%
}_{4})$

We obtain 8 BL-algebras of which 7 are chains, from Remark \ref{Remark_3}.

We deduce that 
\begin{equation*}
\left\vert \mathcal{MV}_{5}\right\vert =\pi (5)+1=1\text{ and }\left\vert 
\mathcal{BL}_{5}\right\vert =9=\left\vert \mathcal{MV}_{5}\right\vert
+\left\vert \mathcal{BL}_{4}\right\vert +\left\vert \mathcal{BL}%
_{3}\right\vert +\left\vert \mathcal{BL}_{2}\right\vert
\end{equation*}%
\begin{equation*}
\left\vert \mathcal{MV}_{5}(c)\right\vert =1\text{ and }\left\vert \mathcal{%
BL}_{5}(c)\right\vert =8.
\end{equation*}
\end{proof}

\textbf{Table 2} present a basic summary of the structure of BL-algebras $L$
with $2\leq n\leq 5$ elements:

\begin{equation*}
\text{\textbf{Table 2:}}
\end{equation*}

\textbf{\ }%
\begin{tabular}{lll}
$\left\vert L\right\vert \mathbf{=n}$ & \textbf{Nr of BL-alg } & \textbf{%
Structure } \\ 
$n=2$ & $1$ & $\left\{ Id(\mathbb{Z}_{2})\text{ (chain, MV)}\right. $ \\ 
$n=3$ & $2$ & $\left\{ 
\begin{array}{c}
Id(\mathbb{Z}_{4})\text{ (chain, MV)} \\ 
Id(\mathbb{Z}_{2})\boxtimes Id(\mathbb{Z}_{2})\text{ (chain)}%
\end{array}%
\right. $ \\ 
$n=4$ & $5$ & $\left\{ 
\begin{array}{c}
Id(\mathbb{Z}_{8})\text{ (chain, MV)} \\ 
Id(\mathbb{Z}_{2}\times \mathbb{Z}_{2})\text{ (MV)} \\ 
Id(\mathbb{Z}_{2})\boxtimes Id(\mathbb{Z}_{4})\text{ (chain)} \\ 
Id(\mathbb{Z}_{4})\boxtimes Id(\mathbb{Z}_{2})\text{ \ (chain)} \\ 
Id(\mathbb{Z}_{2})\boxtimes (Id(\mathbb{Z}_{2})\boxtimes Id(\mathbb{Z}_{2}))%
\text{ (chain)}%
\end{array}%
\right. $ \\ 
$n=5$ & $9$ & $\left\{ 
\begin{array}{c}
Id(\mathbb{Z}_{16})\text{ (chain, MV)} \\ 
Id(\mathbb{Z}_{2})\boxtimes Id(\mathbb{Z}_{8})\text{ (chain)} \\ 
Id(\mathbb{Z}_{2})\boxtimes Id(\mathbb{Z}_{2}\times \mathbb{Z}_{2}) \\ 
Id(\mathbb{Z}_{2})\boxtimes (Id(\mathbb{Z}_{2})\boxtimes Id(\mathbb{Z}_{4}))%
\text{ (chain)} \\ 
Id(\mathbb{Z}_{2})\boxtimes (Id(\mathbb{Z}_{4})\boxtimes Id(\mathbb{Z}_{2}))%
\text{ (chain)} \\ 
Id(\mathbb{Z}_{2})\boxtimes (Id(\mathbb{Z}_{2})\boxtimes (Id(\mathbb{Z}%
_{2})\boxtimes Id(\mathbb{Z}_{2})))\text{ (chain)} \\ 
Id(\mathbb{Z}_{8})\boxtimes Id(\mathbb{Z}_{2})\text{ (chain)} \\ 
(Id(\mathbb{Z}_{4})\boxtimes Id(\mathbb{Z}_{2}))\boxtimes Id(\mathbb{Z}_{2})%
\text{ (chain)} \\ 
Id(\mathbb{Z}_{4})\boxtimes Id(\mathbb{Z}_{4})\text{ (chain)}%
\end{array}%
\right. $%
\end{tabular}

Finally, \textbf{Table 3 }present a summary for the number of MV-algebras,
MV-chains, BL-algebras and BL-chains with $n\leq 5$ elements:

\begin{equation*}
\text{\textbf{Table 3}}
\end{equation*}

\medskip

\textbf{\ \ }%
\begin{tabular}{lllll}
& $n=2$ & $n=3$ & $n=4$ & $n=5$ \\ 
MV-algebras & $1$ & $1$ & $2$ & $1$ \\ 
MV-chains & $1$ & $1$ & $1$ & $1$ \\ 
BL-algebras & $1$ & $2$ & $5$ & $9$ \\ 
BL-chains & $1$ & $2$ & $4$ & $8$%
\end{tabular}%
. 

Cristina Flaut

{\small Faculty of Mathematics and Computer Science, Ovidius University,}

{\small Bd. Mamaia 124, 900527, Constan\c{t}a, Rom\^{a}nia,}

{\small \ http://www.univ-ovidius.ro/math/}

{\small e-mail: cflaut@univ-ovidius.ro; cristina\_flaut@yahoo.com}

\bigskip

Dana Piciu

{\small Faculty of \ Science, University of Craiova, }

{\small A.I. Cuza Street, 13, 200585, Craiova, Romania,}

{\small http://www.math.ucv.ro/dep\_mate/}

{\small e-mail: dana.piciu@edu.ucv.ro, piciudanamarina@yahoo.com}


\medskip

\begin{thebibliography}{99}
\bibitem{[Bl; 53]} Blair, R.L, \textit{Ideal lattices and the structure of
rings}, Trans. Am. Math. Soc., 75(1953), 136--153.

\bibitem{[BP; 02]} Busneag, D., Piciu, D., \textit{Lectii de algebra}, Ed.
Universitaria, Craiova, \textbf{2002}.

\bibitem{[BN; 09]} Belluce, L.P., Di Nola, A., \textit{Commutative rings
whose ideals form an MV-algebra}, Math. Log. Quart., \textbf{55} (5) (2009),
468-486.

\bibitem{[COM; 00]} Cignoli, R.; D'Ottaviano, I.M.L.; Mundici, D., \textit{%
Algebraic Foundations of many-valued Reasoning}. Trends in Logic-Studia
Logica Library 7, Dordrecht: Kluwer Acad. Publ.\textbf{\ 2000}.\medskip

\bibitem{[CHA; 58]} Chang, C.C.,\textit{\ Algebraic analysis of many-valued
logic}, Trans. Amer. Math. Soc. 88(1958), 467-490.

\bibitem{[NL; 03]} Di Nola, A., Lettieri, A., \textit{Finite BL-algebras},
Discrete Mathematics, 269 (2003), 93---112.

\bibitem{[Di; 38]} Dilworth, R.P., \textit{Abstract residuation over lattices%
}, Bull. Am. Math. Soc. 44(1938), 262--268.

\bibitem{[H; 98]} H\'{a}jek, P., \textit{Metamathematics of Fuzzy Logic},
Trends in Logic-Studia Logica Library 4, Dordrecht: Kluwer Academic
Publishers \textbf{1998.\medskip }

\bibitem{[LN; 18]} Heubo-Kwegna, O. A., Lele, C., Nganou, J. B., \textit{%
BL-rings}, Logic Journal of IGPL, 26(3) (2016), 290--299.

\bibitem{[HR; 99]} H\H{o}hle, U., Rodabaugh, S. E., Mathematics of fuzzy
sets:logic, topology and measure theory, Springer, Berlin, \textbf{1999}.

\bibitem{[FP; 22]} Flaut, C., Piciu, D., \textit{Connections between
commutative rings and some algebras of logic}, Iranian Journal of Fuzzy
Systems, \textbf{19} (6) (2022), 93-110.

\bibitem{[I; 09]} Iorgulescu, A., \textit{Algebras of logic as BCK algebras, 
}A.S.E., Bucharest, \textbf{2009}.

\bibitem{[La; 01]} Lam T.Y., \textit{A first course in noncommutative rings}%
, Graduate Texts in Mathematics (2nd ed.). Springer-Verlag, \textbf{2001}.

\bibitem{[TT; 22]} Tchoffo Foka, S. V., Tonga, M., \textit{Rings and
residuated lattices whose fuzzy ideals form a Boolean algebra}, Soft
Computing, 26 (2022) 535-539.

\bibitem{[T; 99]} Turunen, E., \textit{Mathematics Behind Fuzzy Logic}.
Physica-Verlag, \textbf{1999}.

\bibitem{[WD; 39]} Ward, M., Dilworth, R.P., \textit{Residuated lattices},
Trans. Am. Math. Soc. 45(1939), 335--354.
\end{thebibliography}
\end{document}